\documentclass[12pt]{article}

\usepackage{amssymb,amsmath,amsfonts,amsthm,exscale}

\usepackage{color}
\definecolor{marin}{rgb}   {0.,   0.3,   0.7} 
\usepackage[colorlinks=false,pdfborder={0 0 0},%colorlinks,citecolor=marin,linkcolor=marin,urlcolor=marin,
            bookmarksopen,
            bookmarksnumbered
           ]{hyperref}

\definecolor{orange}{rgb}{1.0,0.5,0.0} 
\definecolor{red}{rgb}{0.8,0.2,0.2} 
\definecolor{green}{rgb}{0.3,0.7,0.3} 
\newtheorem{lemma}{Lemma}[section]
\newtheorem{theorem}[lemma]{Theorem}
\newtheorem{remark}[lemma]{Remark}
\numberwithin{equation}{section}

\newcommand{\N}{\mathbb{N}}
\newcommand{\R}{\mathbb{R}}
\newcommand{\C}{\mathbb{C}}
\newcommand{\T}{\mathbb{T}}
\newcommand{\Z}{\mathbb{Z}}
\newcommand{\Mc}{\mathcal{M}}
\newcommand{\Pc}{\mathcal{P}}
\newcommand{\Uc}{\mathcal{U}}
\newcommand{\Vc}{\mathcal{V}}
\newcommand{\Zc}{\mathcal{Z}}
\newcommand{\jb}{{{j}}}
\newcommand{\lb}{{{l}}}
\newcommand{\kb}{{{k}}}
\newcommand\bfomega{{\boldsymbol \omega}}
\newcommand{\Norm}[2]{\|#1\|\left.\vphantom{T_{j_0}^0}\!\!\right._{#2}}   

\DeclareMathOperator{\ImT}{Im}

\begin{document}

\title{Sobolev  stability of plane wave solutions to the cubic nonlinear Schr\"odinger equation on a torus}
%\titlerunning{Sobolev  stability of plane wave solutions to the cubic NLS on a torus}

\author{Erwan Faou\thanks{INRIA and ENS Cachan Bretagne, 
        Avenue Robert Schumann, 
        F-35170 Bruz, France
        ({\tt Erwan.Faou@inria.fr}).}
        \thanks{  D\'{e}partement de math\'{e}matiques et applications, 
          \'{E}cole normale sup\'{e}rieure,
          45 rue d'Ulm,
          F-75230 Paris Cedex 05, France.  }
        \and
        Ludwig Gauckler\thanks{Institut f\"ur Mathematik,
        Technische Universit\"at Berlin,
        Stra{\ss}e des 17.\ Juni 136,
        D-10623 Berlin, Germany
        ({\tt gauckler@math.tu-berlin.de}).}
        \and
        Christian Lubich\thanks{Mathematisches Institut,
        Universit\"at T\"ubingen,
        Auf der Morgenstelle 10,
        D-72076 T\"ubingen, Germany
        ({\tt lubich@na.uni-tuebingen.de}).}
}

\date{Version of 10 October 2012}

\maketitle

\begin{abstract}
It is shown that plane wave solutions to the cubic nonlinear Schr\"od\-inger equation on a torus behave orbitally stable under generic perturbations of the initial data that are small in a high-order Sobolev norm, over long times that extend to arbitrary negative powers of the smallness parameter. The perturbation stays small in the same Sobolev norm over such long times. The proof uses a Hamiltonian reduction and transformation and, alternatively, Birkhoff normal forms or modulated Fourier expansions in time.
\end{abstract}

%%%%%%%%%
\section{Introduction and statement of the result}
%%%%%%%%%

Consider the cubic nonlinear Schr\"odinger  equation (NLS) on the $d$-dimensional torus $\T^{d}=\R^d/(2\pi\Z)^d$, for arbitrary dimension $d\ge 1$,  in the defocusing ($\lambda=1$) or focusing ($\lambda=-1$) case,
\begin{equation}
\label{eq:nls}
i \partial_{t} u = -\Delta u + \lambda |u|^{2} u,\qquad \ x\in \T^d,\, t\in\R.
\end{equation}
For initial data made of a single Fourier mode, $u_*(x,0)=\rho e^{im\cdot x}$, the equation has the plane-wave solution
$u_*(x,t)=\rho e^{i(m\cdot x - \omega t)}$
with
$\omega=|m|^2 + \lambda \rho^2$.
We show that under {\em generic} perturbations of such initial data by functions with small $H^s$ Sobolev norm, for sufficiently large Sobolev exponent $s$, the solution remains essentially localized in the $m$th Fourier mode over very long times, and the perturbation remains small in the same $H^s$ norm.  For the precise formulation of the result we decompose the solution in the Fourier basis, $u(x,t) = \sum_{j \in \Z^d} u_j(t) e^{i j \cdot x}$. 

\begin{theorem}\label{theorem:main}
Let $\rho_0>0$ be such that 
$
1 + 2 \lambda \rho_0^2 > 0,
$
and let $N>1$ be fixed arbitrarily. There exist $s_0>0$,   $C\ge 1$   and a set of full measure $\mathcal{P}$ in the interval $(0,\rho_0]$ such that for every $s \geq s_0$ and  every $\rho \in \Pc$, there exists $\varepsilon_0>0$  such that for every $m \in \Z^d$ the following holds: 
if the initial data $u(\bullet,0)$   are   such that 
$$
\Norm{u(\bullet,0)}{L^2} = \rho
\quad \mbox{and}\quad
\Norm{e^{-i m \cdot \bullet}u(\bullet,0) -  u_{m}(0) }{H^s} = \varepsilon \leq \varepsilon_0, 
$$
then the solution of \eqref{eq:nls} with these initial data satisfies
\begin{equation}
\label{eq:epsN}
\Norm{e^{-i m \cdot \bullet}u(\bullet,t) -  u_{m}(t) }{H^s} \leq   C   \varepsilon \quad \mbox{for}\quad t \leq  \varepsilon^{-N}. 
\end{equation}
\end{theorem}

As will further be shown, this theorem (together with the conservation of $L^2$ norm) implies long-time \emph{orbital stability} in $H^s$: the solution stays in $H^s$ close to the orbit $e^{i\varphi} u_m(0) e^{i m\cdot x}$, $\varphi\in\R$, of the nonlinear Schr\"odinger equation: 
\begin{equation}\label{eq:orb}
\inf_{\varphi\in\R} \Norm{e^{-i m \cdot \bullet}u(\bullet,t) -  e^{i\varphi} u_{m}(0) }{H^s} \le \sqrt{2} C \varepsilon \quad \mbox{for}\quad t \leq  \varepsilon^{-N}. 
\end{equation}
 
% This theorem implies long-time \emph{orbital stability} in $H^s$: the solution stays in $H^s$ close to the orbit $e^{i\varphi} u_m(0) e^{i m\cdot x}$, $\varphi\in\R$, of the nonlinear Schr\"odinger equation, since 
% \[
% \inf_{\varphi\in\R} \Norm{e^{-i m \cdot \bullet}u(\bullet,t) -  e^{i\varphi} u_{m}(0) }{H^s}^2 = \bigl| |u_m(t)| - |u_m(0)| \bigr|^2 + \Norm{e^{-i m \cdot \bullet}u(\bullet,t) -  u_{m}(t) }{H^s}^2
% \]
% and $|u_m(t)|^2 - |u_m(0)|^2 = \sum_{m\ne j\in\Z^d} (|u_j(0)|^2 - |u_j(t)|^2)$ by conservation of the $L^2$ norm, and hence by Theorem~\ref{theorem:main}
% \[
% \inf_{\varphi\in\R} \Norm{e^{-i m \cdot \bullet}u(\bullet,t) -  e^{i\varphi} u_{m}(0) }{H^s} \le 3\varepsilon \quad \mbox{for}\quad t \leq  \varepsilon^{-N}.
% \]

To our knowledge, this is the first long-time orbital stability result as well as long-time high-regularity result for the cubic NLS (\ref{eq:nls}) in dimension $d>1$. It is a contrasting counterpart to the {instability} %and {ill-posedness} 
results for {\it rough} perturbations given by Christ, Colliander \& Tao \cite{CCT03,CCT03a}, Carles, Dumas \& Sparber \cite{CDS2010} and Hani \cite{Hani2011}.

In the 1D case, much is already known about $H^1$-orbital stability of periodic waves for the cubic NLS  through the work by   Zhidkov \cite[Sect.~3.3]{Zhidkov2001}   %Cazenave \& Lions \cite{CL82} 
and Gallay \& Haragus \cite{GH1,GH2}; see also further references therein. 

To obtain the $H^s$-stability result (with $s \gg 1$) presented above, the techniques used are more closely related to 
 long-time stability and high-regularity results by Bambusi \& Gr\'ebert \cite{BG06,Bam07,Greb07} and Gauckler \& Lubich \cite{GL10} for small solutions to {\em modifications} of the periodic cubic NLS (\ref{eq:nls}) by the addition of a convolution term $V\star u$, which eliminates the resonance of the frequencies of the linearization of (\ref{eq:nls}) around 0. Such a resonance-removing modification by a convolution potential was previously studied also by Bourgain \cite{Bo98} and more recently by Eliasson \& Kuksin \cite{EK10}. In this {\em non-resonant} case, $H^s$-stability of small solutions can be proven generically with respect to the external parameter $V$, and in the case where the solutions are analytic, the constant can be optimized to some Nekhoroshev-like estimate as shown by Faou \& Gr\'ebert in \cite{FG}. 

As far as the {\em resonant} case is concerned, i.e., the cubic NLS (\ref{eq:nls}) in dimension $d > 1$, the question of the $H^s$-behavior of solutions is much more delicate and has recently known many advances. Bourgain \cite{Bou96} %and Staffilani \cite{Sta97} 
gives upper bounds of the form 
$$
\Norm{u(t)}{H^s} \leq t^{\alpha(s-1)} \Norm{u(0)}{H^s}   \quad\text{for}\quad t>0  
$$
  in dimension $d=2$ or $d=3$ for sufficiently large $s$, either for the defocusing case $\lambda>0$ or for small initial data in $L^2$ ($d=2$) or $H^1$ ($d=3$). It  
is conjectured by Bourgain in \cite{Bou00} that the previous bound can be refined to subpolynomial growth in time, that is $t^\alpha$ for all $\alpha > 0$. 
Colliander, Keel, Staffilani, Takaoka \& Tao \cite{CKSTT}   and Guardia \& Kaloshin \cite{Guardia2012}   proved,   in dimension $d=2$ and in the defocusing case $\lambda > 0$, for any $\varepsilon$, $M>0$ and $s>1$,   the existence of solutions   $u$   to \eqref{eq:nls} such that $\Norm{u(0)}{H^s} < \varepsilon$ and $\Norm{u(T)}{H^s} > M$ for some time $T > 0$. In a similar direction, Carles \& Faou \cite{Cascades} proved the existence of initial data which propagate energy from low to high modes in an energy cascade that prevents the existence of Nekhoroshev-like stability results with preservation of the actions for small solutions.
% , as can be found in  \cite{BG06,Bam07,Greb07,GL10,FG} for the non-resonant case where a convolution potential $V \star u$ is added. 

In contrast to these results indicating unstable behavior,  Theorem \ref{theorem:main} above indeed gives examples of $H^s$-stability over very long time for the resonant dynamics \eqref{eq:nls}, in a way which looks like known results in non-resonant cases. The peculiarity of the situation is that the plane wave solution creates new frequencies in the dynamics of the resonant equation (\ref{eq:nls}). These new frequencies depend on the initial data  and turn out to be generically non-resonant. This allows us to separate the dynamics between the modes and to prove the preservation of $H^s$-regularity over long time. Note that this does not contradict the results and conjectures about the existence of solutions with Sobolev norm growths, but it shows the existence of many {\em integrable islands} (using an expression due to T.~Kappeler) that are $H^s$-stable under the dynamics of the resonant equation \eqref{eq:nls}. A different resonance-removing effect of the initial value is used by Bourgain \cite{Bou00a} to prove long-time $H^s$-stability of small solutions to \eqref{eq:nls} in dimension $d=1$.

A striking fact is that the solutions we consider are not small in $H^s$, as is traditionally the case when studying non-resonant situations. In particular, the sign of the nonlinearity enters into the statement of the Theorem: while the solution has to be small enough in the focusing case $\lambda < 0$, 
there is no restriction on the size of the initial data in the defocusing case $\lambda > 0$.  

% but we emphasize that such a modification of the equation is not required for the problem considered here, essentially because we do not consider small solutions of (\ref{eq:nls}).

The plan of proof is the following: We first separate the dynamics of the plane wave from the other modes, and we perform some reductions and transformations (in Section 2) to obtain a problem in a Hamiltonian form with new frequencies depending on the initial data. We use the gauge invariance of the equation as well as the $L^2$ norm preservation to eliminate the dynamics of the plane wave itself in the equation. We end up with a situation of a semi-linear Hamiltonian system with non-resonant frequencies, 
for which there are two known techniques to arrive at Theorem~\ref{theorem:main}: {\em Birkhoff normal forms}, which use a sequence of nonlinear canonical coordinate transforms to transform the system to a form from which the dynamical properties can be read off, and {\em modulated Fourier expansions}, which embed the system into a larger modulation system having
%having a Hamiltonian structure with a group invariance property that yields the existence of 
almost-invariants that allow us to infer the desired long-time properties. In Section 3 we reduce Theorem \ref{theorem:main} to a known abstract result on the long-time near-conservation of super-actions, which was proved by Gr\'ebert \cite{Greb07} via Birkhoff normal forms and by Gauckler \cite{Gau10} via modulated Fourier expansions, under differing conditions which we verify for both approaches.

\begin{remark}
Following the approach of \cite{FG}, we could optimize $N \simeq |\log \varepsilon|^\beta$ with $\beta < 1$ in Equation \eqref{eq:epsN}, if the perturbation of the plane wave is analytic. 
\end{remark}

\pagebreak[3]
 
%%%%%%
\section{Reductions and Transformations}
%%%%%%%
\subsection{Reduction to the case \texorpdfstring{$m=0$}{m=0}}

In terms of Fourier coefficients, \eqref{eq:nls} is given by 
\begin{equation}
\label{eq:nlsF}
i \dot u_{j} = |j|^{2} u_{j} + \lambda \sum_{j=j_1-j_2+j_3} u_{j_1} \overline u_{j_2} u_{j_3},
\end{equation}
where $|j|$ denotes the Euclidean norm of $j\in\Z^d$.
With the given $m \in \Z^d$, 
we transform to
\begin{equation}
\label{eq:trsf}
v_j = u_{j + m}e^{ i t (|m|^2 + 2 j \cdot m)}. 
\end{equation}
Note that this transformation preserves the $L^2$ norm.
The equation for $v_j$ is 
$$
(|m|^2 + 2 j \cdot m) v_j  + 
i \dot v_{j} = |j + m|^{2} v_{j} + \lambda \sum_{j=j_1-j_2+j_3} v_{j_1} \overline v_{j_2} v_{j_3}
$$
or equivalently
$$
i \dot v_{j} = |j|^{2} v_{j} + 
\lambda \sum_{j=j_1-j_2+j_3} v_{j_1} \overline v_{j_2} v_{j_3},
$$
so that $v(x,t)=\sum_{j \in \Z^d} v_j(t) e^{i j \cdot x}$ is a solution of \eqref{eq:nls} and is localized in the zero mode if $u$ is localized in the $m$th mode. 
In other words, up to the transformation \eqref{eq:trsf}, we can restrict our attention to the case  $m = 0$. 

\subsection{Elimination of the zero mode}\label{subsection:elimzeromode}
We make the change of variables   $v \mapsto (a,\theta, w)$ with $v = (v_j)_{j\in\Z^d}$, $w=(w_j)_{0\ne j\in\Z^d}$, $0\le a\in\R$ and $\theta\in\R$   defined by 
$$
v_0 = a e^{-i\theta} \quad \mbox{and}\quad v_j = w_j e^{-i\theta} \quad\mbox{for}\quad 0\ne j \in \Z^d.
$$
The equation for $w_j$ reads
\begin{equation}\label{eq:wj}
i\dot w_{j} +\dot\theta w_j = |j|^{2} w_j + \frac{\partial \widehat{P}}{\partial \overline{w}_j} (  a,   w,\overline{w})
\end{equation}
with
\[
\widehat{P}(  a,   w,\overline{w}) = \frac{\lambda}{2} \sum_{j_1+j_2-j_3-j_4 = 0} w_{j_1} w_{j_2} \overline{w}_{j_3} \overline{w}_{j_4},
\]
  where the sum is over indices $j_1,\dots,j_4\in\Z^d$ and where we use the convention $a=w_0=\overline{w}_0$. With the same convention the corresponding equation for $a$ reads
\[
i \dot{a} +\dot\theta a = \lambda \sum_{j_1+j_2-j_3 = 0} w_{j_1} w_{j_2} \overline{w}_{j_3},
\]
and   taking the real part   shows that  
\[
\dot \theta =   \frac{1}{2a} \frac{\partial \widehat{P}}{\partial a} (a,w,\overline{w}).   
\]
Inserting this formula into \eqref{eq:wj} yields an equation for $w_j$ which does not depend on $\theta$ anymore,
\begin{equation}\label{eq:wj2}
i\dot w_{j} = |j|^{2} w_j + \frac{\partial \widehat{P}}{\partial \overline{w}_j} (  a,   w,\overline{w}) + \frac{-w_j}{2   a } \frac{\partial \widehat{P}}{\partial   a  } (  a,   w,\overline{w}).
\end{equation}

Now note that by the conservation of the $L^2$ norm and Parseval's equality we have for all times~$t$, 
\[
\rho^2 = |v_0|^2 + \sum_{j \neq 0} |v_{j}|^2 =   a^2   +   \sum_{j \neq 0} |w_{j}|^2
\]
which means
\begin{equation}
\label{eq:controla}
  a   = \sqrt{\rho^2 - \sum_{j \neq 0} |w_{j}|^2}.
\end{equation}
Hence we can not only forget the dynamics of $\theta$ but also of   $a$  : it will be controlled by the $w_j$, $j\ne 0$, using \eqref{eq:controla}. We arrive at a system of differential equations for a reduced set of variables
\[
z = (z_j)_{j\in\Zc} := (w_j)_{j\in\Zc} \qquad\text{with}\qquad \Zc := \Z^d \setminus \{ 0 \}.
\]

\subsection{The reduced Hamiltonian system}
Considering   $a$   as a function of $w_j$ and $\overline{w}_j$, $j\in\Zc$, given by \eqref{eq:controla} we have
\[
\frac{\partial   a  }{\partial \overline{w}_j} = \frac{-w_j}{2   a  },
\]
and therefore the equations \eqref{eq:wj2} for $z = (w_j)_{j\in\Zc}$ are Hamiltonian,
\begin{equation}\label{eq:ode-wj}
i \dot z_j = \frac{\partial \widetilde{H}}{\partial\overline{z}_j}(z,\overline{z}), \qquad  j \in \Zc ,
\end{equation}
with the real-valued Hamiltonian function
\[
\widetilde{H}(z,\overline{z}) = \sum_{j} |j|^2 z_{j}\overline{z}_{j} + \widetilde{P}(z,\overline{z}) \qquad\text{with}\qquad \widetilde{P}(z,\overline{z}) = \widehat{P}(  a,   w,\overline{w}).
\]
This Hamiltonian function takes the form
\begin{equation}\label{eq:Htilde}
\begin{split}
\widetilde{H}(z,\overline{z}) &= \lambda \rho^4 + \sum_{j} (|j|^2 + \lambda\rho^2) z_{j}\overline{z}_{j} + \frac{\lambda}{2} \rho^2 \sum_{j} \overline{z}_{j} \overline{z}_{-j} + \frac{\lambda}{2} \rho^2 \sum_{j} z_{j} z_{-j}\\
 &\quad+ \frac{\lambda}{2} \sum_{j_1+j_2-j_3-j_4=0} z_{j_1} z_{j_2} \overline{z}_{j_3} \overline{z}_{j_4} - \frac{3\lambda}{2} \Bigl(\sum_{j_1}z_{j_1} \overline{z}_{j_1}\Bigr) \Bigl(\sum_{j_2}z_{j_2} \overline{z}_{j_2}\Bigr)\\
 &\quad- \frac{\lambda}{2} \Bigl(\sum_{j_1}z_{j_1}z_{-j_1}\Bigr) \Bigl(\sum_{j_2}z_{j_2} \overline{z}_{j_2}\Bigr) - \frac{\lambda}{2} \Bigl(\sum_{j_1}\overline{z}_{j_1}\overline{z}_{-j_1}\Bigr) \Bigl(\sum_{j_2}z_{j_2} \overline{z}_{j_2}\Bigr)\\
 &\quad + \lambda \Bigl(\sum_{j_1+j_2-j_3=0}\!\! z_{j_1} z_{j_2} \overline{z}_{j_3} + \!\!\sum_{j_1-j_2-j_3=0}\!\! z_{j_1} \overline{z}_{j_2} \overline{z}_{j_3} \Bigr) \sqrt{\rho^2 - \sum_{j_4}z_{j_4} \overline{z}_{j_4}}.
\end{split}\end{equation}

Expanding $\sqrt{\rho^2 - x}$ into a convergent power series for $|x| < \rho^2$,
we can write the Hamiltonian \eqref{eq:Htilde} as the infinite sum 
$$
\widetilde{H}(z,\overline{z}) = \lambda \rho^4 + \sum_{r \geq 2} \widetilde H_r(z,\overline{z})
$$
where $\widetilde H_r(z,\overline{z})$ is a homogeneous polynomial of degree $r$ in terms of $(z_j,\overline z_j)$, which is of the form
$$
\widetilde H_r(z,\overline z) = \sum_{p + q = r} \,\sum_{\substack{(\kb,\lb)  \in \Zc^p \times \Zc^q \\\Mc(\kb,\lb) = 0}}\widetilde H_{\kb \lb} \,z_{k_1} \cdots z_{k_p} \overline z_{l_1} \cdots \overline z_{l_q} 
$$
where 
\begin{equation}\label{eq:momentum}
\Mc(\kb,\lb) = k_1 + \ldots + k_p - l_1 - \ldots -l_q
\end{equation}
denotes the {\em momentum} of the multi-index $(\kb,\lb)$. We note that the Taylor expansion of $\widetilde{H}$ contains only terms with zero momentum, and its coefficients satisfy the bound 
\begin{equation}
\label{eq:bound-tilde}
 |\widetilde H_{\kb \lb} | \leq \widetilde M\,{\widetilde L}^{p+q}  \quad \hbox{ for all } (\kb,\lb)  \in \Zc^p \times \Zc^q, 
\end{equation}
where $\widetilde M$ and $\widetilde L$  depend on $\rho$.

%%%%%%%
\subsection{Diagonalization  and non-resonant frequencies}
%%%%%%%

We study now the linear part of the system (\ref{eq:ode-wj}). As we will see, its eigenvalues are non-resonant for almost all parameters~$\rho$. Moreover, we can control the diagonalization of this linear operator.

The linear part in the differential equation for $z_j$ is $(|j|^2+ \lambda \rho^2) z_j + \lambda \rho^2 \overline z_{-j}$. 
On taking the equation for $z_j$ together with that for $\overline z_{-j}$, we are thus led to consider the matrix (for $n=|j|^2$)
$$
A_n = \begin{pmatrix}
n + \lambda \rho^2  & \lambda \rho^2\\
- \lambda \rho^2 & -n - \lambda \rho^2
\end{pmatrix}.
$$

\begin{lemma}\label{lemma:diag}
For all $n \geq 1$, the matrix 
$A_n$ is diagonalized by a $2\times 2$ matrix $S_n$ that is real symplectic and hermitian and has condition number smaller than $2$:
$$
S_n^{-1} A_n S_n = \begin{pmatrix} \Omega_n & 0 \\ 0 & -\Omega_n\end{pmatrix}
\quad\hbox{with}\quad
\Omega_n = \sqrt{n^2 + 2n\lambda \rho^2}.
$$
\end{lemma}
\begin{proof}
We obtain
$$
S_n = \frac{1}{\sqrt{(\Omega_n + n)(\Omega_n + n + 2 \lambda \rho^2)}}\begin{pmatrix}
n + \lambda\rho^2 + \Omega_n & - \lambda\rho^2 \\ 
- \lambda \rho^2 & n +  \lambda \rho^2 + \Omega_n
\end{pmatrix}
$$
and
$$
S_n^{-1} = \frac{1}{\sqrt{(\Omega_n + n)(\Omega_n + n + 2 \lambda \rho^2)}}\begin{pmatrix}
n + \lambda\rho^2 + \Omega_n & \lambda\rho^2 \\ 
 \lambda \rho^2 & n +  \lambda \rho^2 + \Omega_n
\end{pmatrix},
$$
and the statements of the lemma then follow by direct verification.
\end{proof}

Note that the condition $1+2\lambda\rho^2>0$ in Theorem~\ref{theorem:main} ensures that all the eigenvalues $\Omega_n$ are real, or equivalently, that the linearization of the system (\ref{eq:ode-wj}) at $0$ is stable.
The frequencies $\Omega_n$ turn out to satisfy Bambusi's  non-resonance inequality \cite{Bam03,BG06} for almost all norm parameters $\rho>0$.

\begin{lemma}\label{lem:non-res}
Let $r>1$ and $\rho_0>0$ with $1+2\lambda \rho_0^2>0$.
There exist $\alpha=\alpha(r)>0$ and a set of full Lebesgue measure ${\mathcal P}\subset (0,\rho_0]$ such that for every $\rho\in{\mathcal P}$ there is a $\gamma>0$ such that the following non-resonance condition is satisfied: for all positive integers $p,q$ with $p+q\le r$ and for all $m = (m_1,\ldots,m_p) \in \N^p$ and $n = (n_1,\ldots,n_q)\in \N^q$,
\begin{equation}
\label{eq:nrc}
|\Omega_{m_1} + \ldots + \Omega_{m_p} -\Omega_{n_1} -\ldots - \Omega_{n_q}| \geq \frac{\gamma}{\mu_3(m,n)^{\alpha}},
\end{equation}
except if the frequencies cancel pairwise. Here, $\mu_3(m,n)$ denotes the third-largest among the integers $m_1,\ldots,m_p,n_1,\ldots,n_q$.
\end{lemma}
\begin{proof}
  
The proof is similar to the one given in \cite[Section 5.1]{BG06} for the frequencies $\sqrt{n^2+\rho}$, $n\in\N$.

(a) Considering the frequencies as functions of $\sigma=\rho^2$, $\Omega_n=\Omega_n(\sigma)$, we have
\begin{equation}\label{eq:derfreq}
\frac{d^k}{d\sigma^k} \Omega_{n} = c_k (n\lambda)^k \Omega_n^{1-2k}
\end{equation}
with $c_k = (3-2k) c_{k-1}$ for $k\ge 1$ and $c_0=1$. This shows that, for $0<n_0<n_1<\dots<n_K$, the matrix
\[
\Bigl( \frac{d^k}{d\sigma^k} \Omega_{n_l} \Bigr)_{k,l=0}^K
\]
is of Vandermonde form with determinant
\begin{align*}
\det\Bigl( \frac{d^k}{d\sigma^k} \Omega_{n_l} \Bigr)_{k,l=0}^K &= c_0\dotsm c_K \Omega_{n_0} \dotsm \Omega_{n_K} \prod_{k<l} \Bigl(\frac{n_l\lambda}{\Omega_{n_l}^2} - \frac{n_k\lambda}{\Omega_{n_k}^2}\Bigr)\\
 &= c_0\dotsm c_K \Omega_{n_0} \dotsm \Omega_{n_K} \prod_{k<l} \frac{(n_k-n_l)\lambda}{(n_k+2\lambda\sigma)(n_l+2\lambda\sigma)}.
%\frac{n_k\lambda\Omega_{n_k}^2 - n_l\lambda\Omega_{n_l}^2}{\Omega_{n_k}^2\Omega_{n_l}^2}.
\end{align*}
We infer the lower bound
\begin{equation}\label{eq:detfreq}
\Bigl| \det\Bigl( \frac{d^k}{d\sigma^k} \Omega_{n_l} \Bigr)_{k,l=0}^K\Bigr| \ge \frac{C}{n_K^{K(K+1)}}
\end{equation}
with a constant $C$ depending on $K$, a lower bound of $1+2\lambda\rho^2$ and an uper bound of $\rho$ (see also \cite[Lemma 5.1]{BG06}).

(b) Having established the lower bound \eqref{eq:detfreq} we can proceed as in \cite[Lemmas/Corollaries 5.2--5.6]{BG06} to show that for $\alpha$ sufficiently large compared to $r$ the set
\begin{align*}
\mathcal{P}_{\gamma} = \Bigl\{ \, \sigma\in [0,\rho_0^2] : \Bigl|&\sum_{n=1}^N k_n\Omega_n + m \Bigr| \ge \frac{\gamma}{N^\alpha} \text{ for all } N\ge 1, \text{ all } m\in\Z\\
&\text{and all $0\ne k\in\Z^N$ with } \sum_{n=1}^N |k_n| \le r \, \Bigr\}
\end{align*}
has large Lebesgue measure,
\[
|[0,\rho_0^2] \setminus \mathcal{P}_{\gamma}| \le C \gamma^{1/r}.
\]
In fact, for fixed $k\in\Z^n$, the lower bound \eqref{eq:detfreq} and equation \eqref{eq:derfreq} imply that $|d^l/d\sigma^l \sum_{n=1}^N k_n\Omega_n| \ge C N^{-K(K+1)}$ for at least one derivative $l\le \sum_n |k_n|$ (see Lemma 5.2 and Corollary 5.3 in \cite{BG06}). This estimate of a derivative can be lifted to a corresponding estimate of $|\sum_{n=1}^N k_n\Omega_n + m|$ with $m\in\Z$ for many values of $\sigma$ (see Lemma 5.4 and Corollary 5.5 in \cite{BG06}). Finally one considers the intersection over all $k\in\Z^N$, all $N\ge 1$ and all $m\in\Z$ (see Lemma 5.6 in \cite{BG06}).

(c) Part (b) establishes the non-resonance condition stated in the lemma with the largest integer among $m_1,\ldots,m_p,n_1,\ldots,n_q$ instead of the third-largest. In order to get the third-largest index instead we proceed as in Lemma 5.7 of \cite{BG06} using the asymptotic behaviour 
 
\begin{equation}\label{eq:freqasymptotic}
\Omega_n = n + \lambda \sigma - \frac{\sigma^2}{2 (n + \lambda \sigma)} 
+ \mathcal{O}\Bigl(\frac{1}{n^2}\Bigr)
\end{equation}
  of the frequencies. Indeed, considering a linear combination of frequencies as in \eqref{eq:nrc} we distinguish three cases depending on the size of the largest index $\mu_1(m,n)$ and of the second largest index $\mu_2(m,n)$ among the integers $m_1,\ldots,m_p,n_1,\ldots,n_q$ in comparison to the third-largest index $\mu_3(m,n)$: Either we can bound $\mu_1(m,n)$ in terms of $\mu_3(m,n)$ and apply part (b) directly, or we can consider $\Omega_{\mu_1(m,n)} - \Omega_{\mu_2(m,n)}$ by \eqref{eq:freqasymptotic} as a small perturbation of an integer number (taking the role of $m$ in $\mathcal{P}_{\gamma}$ of part (b)), or in the case that $\Omega_{\mu_1}$ and $\Omega_{\mu_2}$ have the same sign in \eqref{eq:nrc} and $\mu_2(m,n)$ is large compared to $\mu_3(m,n)$ we can exclude any near-resonance.
%We can then use the proof of \cite[Section 5.1]{BG06} (see in particular (5.13) in \cite{BG06}).  
\end{proof}

%%%%%%%
\subsection{The transformed Hamiltonian system}
%%%%%%%

Applying the symplectic\footnote{  The transformation is symplectic with respect to the standard form $\omega(\xi,\xi') = \sum_j \ImT(\xi_j\overline{\xi}_j')$ as a direct calculation using the real symplecticity of $S_n$ ($\det(S_n)=1$) shows.  } linear transformation 
$$
\label{eq:transfs}
\begin{pmatrix}
z_j \\ \overline z_{-j} 
\end{pmatrix}
= S_{n} \begin{pmatrix}
\xi_j \\ \overline \xi_{-j} 
\end{pmatrix}
\quad  \hbox{ for } j \in \Zc \hbox{ and } n = |j|^2, 
$$
with the matrices $S_n$
of Lemma~\ref{lemma:diag},
to the Hamiltonian system (\ref{eq:ode-wj}) of equations for $z_j$, we end up with a Hamiltonian system 
$$
i \frac{d}{dt} \xi_j(t) = \frac{\partial H}{\partial \overline{\xi}_j}(\xi(t),\overline{\xi(t)}), \qquad j\in\Zc = \mathbb{Z}^d\setminus\{0\},
$$
with the real-valued Hamilton function
\[
H(\xi,\overline{\xi}) = \widetilde{H}(z,\overline{z}),
\]
with $\widetilde{H}$ of (\ref{eq:Htilde}). This Hamiltonian is of the form 
\begin{equation}\label{eq:HP}
H(\xi,\overline{\xi}) = \sum_{j\in\Zc} \omega_j|\xi_j|^2 + P(\xi,\overline{\xi}),
\end{equation}
where the frequencies are $\omega_j = \Omega_{n} $ for $|j|^2=n$ with $\Omega_n= \sqrt{n^2+2n\lambda\rho^2}$, and the non-quadratic term $P$ is of the form
\begin{equation}\label{eq:P}
P(\xi,\overline{\xi}) = \sum_{p+q\ge 3}\, \sum_{\substack{k\in\Zc^p,\, l\in\Zc^{q}\\ \Mc(\kb,\lb)=0} }H_{\kb\lb} \,\xi_{k_1}\dotsm \xi_{k_p} \,\overline{\xi}_{l_1}\dotsm \overline{\xi}_{l_q},
\end{equation}
where the sum is still only over multi-indices with zero momentum (\ref{eq:momentum}), since the transformation mixes only terms that give the same contribution to the momentum. From (\ref{eq:bound-tilde}) and Lemma~\ref{lemma:diag} we obtain the following bound for the Taylor coefficients.

\begin{lemma}\label{lem:H}
There exist $M>0$ and $L>0$ such that for all positive integers $p,q$ with $p+q\ge 3$ the coefficients in (\ref{eq:P}) are bounded by
$$
|H_{kl}| \le M \, L^{p+q}    \quad\hbox{ for all }\ k\in\Zc^p, l\in \Zc^q.
$$
\end{lemma}

The Hamiltonian equations of motion are now 
\begin{equation}\label{eq:hammotion}
i \frac{d}{dt} \xi_j(t) = \omega_j \xi_j(t) + \frac{\partial P}{\partial \overline{\xi}_j}(\xi(t),\overline{\xi(t)}), \qquad j\in\Zc,
\end{equation}
where the nonlinearity  is of the form
\begin{equation}\label{eq:nonlin}
\frac{\partial P}{\partial\overline{\xi}_j}(\xi,\overline{\xi}) = \sum_{p+q\ge 2} \,\sum_{\substack{k\in\Zc^p,\, l\in\Zc^{q}\\ \Mc(\kb,\lb)=j} } P_{j,k,l} \,\xi_{k_1}\dotsm \xi_{k_p} \,\overline{\xi}_{l_1}\dotsm \overline{\xi}_{l_q}
\end{equation}
with $P_{j,k,l} $ an integral 
%\Rd (integral multiple?) \Bk  
multiple (at most $(q+1)$ times) of $H_{k,(l,j)}$. 

Note that after the change of variables, the weighted $\ell^2$-norm
$$
\Norm{\xi}{s} = \Bigl( \sum_{j \in \Zc} |j|^{2s} |\xi_j|^2 \Bigr)^{\frac{1}{2}}
$$
of the sequence $\xi$ is equivalent to the Sobolev norm $\Norm{e^{-im\cdot   \bullet  }u - u_m}{H^s}$ of the corresponding function $u$,
\begin{equation}\label{eq:norms}
\hat{c} \, \Norm{\xi}{s} \le \Norm{e^{-im\cdot   \bullet  }u - u_m}{H^s} \le \hat{C} \, \Norm{\xi}{s}
\end{equation}
with positive constants depending on $\lambda$ and $\rho$. In particular, under the assumptions of Theorem~\ref{theorem:main}, the system~\eqref{eq:hammotion} has small initial values whose $\ell^2_s$ norm is of order $\varepsilon$.
%\Rd Note also that $\Norm{\xi}{L^2} = \Norm{u}{L^2}$. \Bk

%%%%%%%%%%%%%%%%%%%%%%%%%%%%%%%%%%%%%%%%%%%
\section{Long-time near-conservation of super-actions}
%%%%%%%%%%%%%%%%%%%%%%%%%%%%%%%%%%%%%%%%%%%

In this section we give the proof of Theorem~\ref{theorem:main}. 
After the transformations of the previous section, we can verify that the conditions required to apply  existing results on the long-time near-conservation of so-called super-actions are fulfilled. A transformation back to the original variables then gives Theorem~\ref{theorem:main}.

\subsection{Super-actions}

Without the nonlinearity $\frac{\partial P}{\partial \overline{\xi}_j}$ in \eqref{eq:hammotion}, the \emph{actions}
\[
I_j(\xi,\overline{\xi}) = |\xi_j|^2,\quad j\in\Zc,
\]
would be exactly conserved along solutions of \eqref{eq:hammotion}. In the presence of the nonlinearity and in view of the partial resonance $\omega_j=\Omega_n$ for all $j\in \Zc$ with $|j|^2=n$ and the non-resonance of the $\Omega_n$ as given by 
Lemma~\ref{lem:non-res}, there remains  long-time near-conservation of \emph{super-actions}
\begin{equation}
\label{eq:supact}
J_n(\xi,\overline{\xi}) = \sum_{|j|^2=n} I_j(\xi,\overline{\xi}), \quad n\in\N,
\end{equation}
along solutions of \eqref{eq:hammotion} provided that the initial value is small. The precise result in our situation is the following.

\begin{theorem}[Long-time near-conservation of super-actions]\label{theorem:superactions} 
Fix $N>1$ arbitrarily.
For every $\rho_0>0$ such that $1 + 2 \lambda \rho_0^2 > 0$ there exist $s_0>0$ and a set of full measure $\mathcal{P}$ in the interval $(0,\rho_0]$ such that for  all $s \geq s_0$ and $\rho \in \Pc$  the following holds: There exist $\varepsilon_0>0$ and $C$ such that for small initial data satisfying
$$
\Norm{\xi(0)}{s}\le \varepsilon \le\varepsilon_0,
$$
the super-actions of the solution of \eqref{eq:hammotion} starting with $\xi(0)$ at $t = 0$ are nearly conserved,
$$
\sum_{n\ge 1} n^{s} \frac{|J_n(\xi(t),\overline{\xi(t)}) - J_n(\xi(0),\overline{\xi(0)})|}{\varepsilon^2} \le C\varepsilon^{\frac{1}{2}},
$$
over long times
$$
0\le t\le \varepsilon^{-N}.
$$
\end{theorem}

Since
$
\| \xi \|_s^2 = \sum_{n\ge 1} n^{s} J_n(\xi,\overline\xi),
$
this theorem implies that $\Norm{\xi(t)}{s}$ stays of order $\varepsilon$ over long times $t\le \varepsilon^{-N}$. When we transform this result back to the solution $u$ of \eqref{eq:nls}, we immediately get Theorem~\ref{theorem:main} on using \eqref{eq:norms}.

There are two entirely different approaches to prove Theorem~\ref{theorem:superactions}, \emph{Birkhoff normal forms} and \emph{modulated Fourier expansions}. Both approaches will be outlined in the following subsections. Each proof relies on a non-resonance condition on the frequencies $\omega_j$ describing the linear part in \eqref{eq:hammotion}, a regularity condition on the nonlinearity in \eqref{eq:hammotion} and a condition on the interaction of modes (zero-momentum condition). Based on Lemmas \ref{lem:non-res} and \ref{lem:H}, these assumptions will be verified in the following subsections, separately for each approach since they are not exactly the same for both proofs. Once the conditions are verified, we can directly apply results from Bambusi \& Gr\'ebert \cite{Bam07,Greb07}  (using Birkhoff normal forms) and Gauckler \cite{Gau10} (using modulated Fourier expansions) to obtain Theorem~\ref{theorem:superactions}.

%%%%%%%%%%%%%%%%%%%%%%%
\subsection{Proof of Theorem~\ref{theorem:superactions} via Birkhoff normal forms}\label{subsection:bnf}
%%%%%%%%%%%%%%%%%%%%%%%

We follow the Birkhoff normal form approach as developed in \cite{Bam03,Bam07,BG06,Greb07}. 
We verify that the assumptions of \cite[Theorem 7.2]{Greb07} are fulfilled by the system \eqref{eq:hammotion}. 

%%%%%%%%%%%%%%%%
\subsubsection{Regularity of the nonlinearity}
%%%%%%%%%%%%%%%%

For multi-indices $\kb = (k_1,\ldots, k_p) \in \Zc^p$ and $\lb = (l_1,\ldots, l_{  q  }) \in \Zc^q$
 we denote by $\mu_i(\kb,\lb)$ the $i$-th largest integer among
 $|k_1|, \ldots, |k_p|,|l_1|, \ldots, |l_q|$, so that $\mu_1(\jb) \geq \mu_2(\jb) \geq\mu_3(\jb)\geq \cdots$. 
Moreover, for a given positive radius $r$, we set 
$$
B_s(r) = \{ (\xi,\overline \xi) \in \C^{\Zc} \times \C^{\Zc} \, :Ê\, \Norm{\xi}{s} \leq r \}. 
$$
To apply Theorem 7.2 in \cite{Greb07}, the Hamilton function
\[
H=H_0+P
\]
with
\begin{eqnarray*}
H_0(\xi,\overline\xi) &=& \sum_{j\in\Zc} \omega_j|\xi_j|^2 \\
P(\xi,\overline{\xi}) &=& \sum_{p+q\ge 3}\, \sum_{\substack{k\in\Zc^p,\, l\in\Zc^{q}\\ \Mc(\kb,\lb)=0} }H_{\kb\lb} \,\xi_{k_1}\dotsm \xi_{k_p} \,\overline{\xi}_{l_1}\dotsm \overline{\xi}_{l_q},
\end{eqnarray*}
needs to satisfy a non-resonance condition on the frequencies, as provided by Lemma~\ref{lem:non-res}, and the following conditions on the non-quadratic part
 (see \cite[Definition 4.4]{Greb07}):
\begin{itemize}
\item[(H1)] There exists $s_0 \geq 0$ such that for all $s \geq s_0$, there exists $r>0$ such that $P \in \mathcal{C}^{\infty}(B_s(r), \C)$.
\item[(H2)] The Taylor coefficients $H_{\kb\lb}$ satisfy the following property: for all 
$(\kb,\lb) \in \Zc^{p} \times \Zc^{q}$ we have $\overline H_{\kb \lb} = H_{\lb \kb}$, and 
for all $p,q$ with $p+q\ge 3$, there exists $\nu \geq 0$ such that for every $N \in \N$, there exists $C$ depending on $N$, $p$ and $q$ such that for all $(\kb,\lb)  \in \Zc^p \times \Zc^q$,
\begin{equation}
\label{eq:boundmu}
 | H_{\kb \lb}| \leq C \mu_3(\kb,\lb)^\nu  \left( \frac{\mu_3(\kb,\lb)}{\mu_3(\kb,\lb) + \mu_1(\kb,\lb) - \mu_2(\kb,\lb) }\right)^N .
\end{equation}
\end{itemize}

The bound \eqref{eq:boundmu}, used in many works on normal forms applied to nonlinear PDEs - see  \cite{DS04,DS06,Bam03,Greb07,BDGS}- implies in particular that the nonlinearity acts on the ball $B_s(r)$. Moreover, it is preserved by the Poisson bracket of two functions and by the normal form construction under a non-resonance condition implying a control of the small denominator by the third largest integer. 

We now show that (H1) and (H2) are implied by the coefficient estimates of Lemma~\ref{lem:H} together with the fact that the Hamiltonian has only terms with zero momentum. 
Let us consider a fixed multi-index $(\kb,\lb)$ satisfying $\Mc(\kb,\lb) = 0$. 
Following the proof of \cite[Lemma 5.2]{Greb07}, we see that we always have
$$
|Ê\mu_1(\kb,\lb) - \mu_2(\kb,\lb) |Ê\leq |Ê\Mc(\kb,\lb) | + \sum_{n = 3}^{p+q} \mu_n(\kb,\lb) \leq (p + q - 2) \mu_{3}(\kb,\lb). 
$$
From this relation, we infer 
$$
 \left( \frac{\mu_3(\kb,\lb)}{\mu_3(\kb,\lb) + \mu_1(\kb,\lb) - \mu_2(\kb,\lb) }\right)^N  \geq (p + q - 1)^{-N}   .   
$$
Using  Lemma~\ref{lem:H}, we thus see that the coefficients $H_{\kb \lb} $ satisfy the bound \eqref{eq:boundmu} with the constant $C = M L^{p+q}(p + q - 1)^{N}$ and $\nu = 0$. This yields (H2). The assertion (H1) results from the fact that $\sqrt{\rho^2 - \sum_{j \in \Zc} z_j \overline z_j}$ is analytic on $B_s(r)$ for $r < \rho$, and that the monomials with zero momentum terms define  a smooth Hamiltonian as soon as $s_0 > d/2$ (see for instance \cite{BG06}). This ensures that $P \in \mathcal{C}^{\infty}(B_s(r), \C)$.

%%%%%%%%%%%%%%%%
\subsubsection{A normal form result}
%%%%%%%%%%%%%%%%

For a given Hamiltonian $K \in \mathcal{C}^{\infty}(B_s(r), \C)$ satisfying $K(\xi,\overline \xi) \in \R$ we denote by $X_K(\xi,\overline \xi)$ the Hamiltonian vector field
$$
X_K(\xi,\overline \xi)_j = \Big( i \frac{\partial K}{\partial \xi_j}, -i \frac{\partial K}{\partial \overline\xi_j}\Big), \qquad j \in \Zc,
$$
associated with the Poisson bracket
$$
\{K,G\} = i \sum_{j \in \Zc} \frac{\partial K}{\partial \xi_j}\frac{\partial G}{\partial \overline \xi_j} - \frac{\partial K}{\partial \overline \xi_j}\frac{\partial G}{\partial \xi_j},
$$
which is well defined for Hamiltonian functions $K$ and $G$ in the class of Hamiltonians defined above. 

We are now ready to apply Theorem 7.2 of \cite{Greb07} to the 
Hamiltonian (\ref{eq:HP}). We obtain the following result: 
\begin{theorem}
Let  $\rho$ be in the set $\Pc$ of full measure as given by Lemma \ref{lem:non-res} for some $N\geq 3$. There exists $s_0$ and for any $s \geq s_0$ there exist two neighborhoods $\Uc$ and $\Vc$ of the origin in $B_s(\rho)$ and an analytic  canonical transformation $\tau: \Vc \to \Uc$ which puts $H = H_0 + P$ in normal form up to order $N$, i.e., 
$$
H \circ \tau = H_0 + Z + R
$$
where 
\begin{itemize}
\item[(i)] $Z$ is a polynomial of degree $N$ which commutes with all the $J_n$, $n\geq 1$, i.e., $\{ Z,J_n\} = 0$ for all $n \geq 1$, 
\item[(ii)] $R \in \mathcal{C}^{\infty}(\Vc,\R)$ and $\Norm{X_R(\xi,\overline \xi)}{s} \leq C_s \Norm{\xi}{s}^N$ for $\xi \in \Vc$, 
\item[(iii)] $\tau$ is close to the identity: $\Norm{\tau(\xi,\overline \xi) - (\xi,\overline \xi)}{s} \leq C_s \Norm{\xi}{s}^2$ for all $\xi \in \Vc$. 
\end{itemize}

\end{theorem}

Let us recall the principle underlying the proof of this result: the construction of the transformation $\tau$ is made by induction by cancelling iteratively the polynomials of growing degree in the Hamiltonian $H_0 + P$. As $\tau$ is determined as the flow at time one of a polynomial Hamiltonian $\chi = \sum_{n \geq 3} \chi_n$, where $\chi_n$ are homogeneous polynomials of degree $n$, we are led to solve by induction homological equations of the form 
$$
\{ H_0, \chi_n\} = Z_n + Q_n
$$
where $Z_n$ is the $n$-th component of the normal form term, and $Q_n$ a homogeneous polynomial of degree $n$ depending on $P$ and   on   the terms constructed at the previous iterations. Writing the equation in terms of coefficients, this equation can be written in the form
$$
(\omega_{k_1} + \ldots + \omega_{k_p} - \omega_{l_1} - \ldots - \omega_{l_q}) \chi_{\kb\lb} = Z_{\kb\lb} + Q_{\kb\lb}
$$
where $(\kb,\lb) \in \Zc^{p}\times\Zc^q$ with $p+q=n$. Using the non-resonance condition \eqref{eq:nrc}, we see that we can solve this equation for $\chi_{kl}$ 
and set $Z_{\kb\lb} = 0$ 
without losing too much regularity (i.e., $\chi$ will satisfy (H2) for some $\nu$), except for the multi-indices $(\kb,\lb)$ having equal length $p = q$ and after permutation, $|k_1|^2 = |l_1|^2, \ldots, |k_p|^2 = |l_p|^2$. This yields that the normal form term $Z$ contains only terms of the form $%\xi_k \overline \xi_l=
\xi_{k_1}\cdots\xi_{k_p}\overline\xi_{l_1}\cdots\overline\xi_{l_q}$ with $p=q$ and $|k_1|^2 = |l_1|^2, \ldots, |k_p|^2 = |l_p|^2$. We then check that these terms Poisson-commute with $J_n$ for all $n$, and hence
$\{ÊJ_n, Z\} = 0$ for all $n$.

The proof of Theorem \ref{theorem:superactions} can then be done using  $\Norm{\xi}{s}^2 = \sum_{n \geq 1} n^{s} J_n (\xi,\overline \xi)$ and following the proof of Corollary 4.9 in \cite{Greb07}.

%%%%%%%%%%%%%%%%%%%%%%%%%%%%%%%%%%%%%%%%%%%
\subsection{Proof of Theorem~\ref{theorem:superactions} via modulated Fourier expansions}
%%%%%%%%%%%%%%%%%%%%%%%%%%%%%%%%%%%%%%%%%%%
 
Modulated Fourier expansions in time have been developed as a technique for analysing weakly nonlinear oscillatory systems over long times, both continuous and discrete systems, in finite and infinite dimensions. There are
two ingredients:
\begin{itemize}
 \item a solution approximation over short time (the modulated Fourier expansion properly speaking)
 \item almost-invariants of the modulation system.
\end{itemize}
The technique can be viewed as embedding the original system in a larger modulation system that turns out to have a Hamiltonian/Lagran\-gian structure with an invariance property from which results on the long-time behaviour can be inferred.

Modulated Fourier expansions 
%are an analytical technique for studying long-time properties of nonlinear perturbations to oscillatory linear systems. They 
were originally introduced in \cite{HL00} to explain the long-time behaviour of numerical methods for oscillatory ordinary differential equations; see also \cite[Chap.\, XIII]{HLW06}. In \cite{CHL08,GHLW11} and \cite{GL10} they were used to study long-time properties of small solutions of nonlinear wave equations and nonlinear Schr\"odinger equations with an external potential, respectively, and in \cite{Gau10} for general classes of Hamiltonian partial differential equations. For a recent review of the technique and various of its uses see \cite{HL12}.

Let us assume for the moment that there are no resonances among the frequencies, in particular $\omega_j\ne\omega_{j'}$ also for $|j|=|j'|$. A modulated Fourier expansion of the solution $\xi$ of \eqref{eq:hammotion} is an approximation of $\xi$ in terms of products of propagators $e^{-i\omega_jt}$ for the linear equation with slowly varying coefficient functions,
\begin{equation}\label{eq:mfe}
\xi_j(t)\approx \widetilde\xi_j(t)=\sum_{\mathbf{k}} z^{\mathbf{k}}_j(\varepsilon t) e^{-i(\mathbf{k}\cdot\bfomega)t}, \qquad j\in\Zc,
\end{equation}
where the sum runs over a finite set of sequences of integers $\mathbf{k} = (\mathbf{k} (\ell))_{\ell\in\Zc}\in\Z^{\Zc}$ with finitely many nonzero entries, and where $\mathbf{k}\cdot\bfomega = \sum_{\ell\in\Zc}\mathbf{k} (\ell)\omega_\ell$. 

Inserting the ansatz \eqref{eq:mfe} into the equations of motion \eqref{eq:hammotion} and equating terms with the same exponential $e^{-i(\mathbf{k}\cdot\bfomega)t}$ leads to the \emph{modulation system}
\begin{align*}
&i\varepsilon\dot{z}^{\mathbf{k}}_j + (\mathbf{k}\cdot\bfomega) z^{\mathbf{k}}_j\\
&\qquad = \omega_jz^{\mathbf{k}}_j + \sum_{p+q\ge 2}^{\infty} \sum_{\substack{\mathbf{k}^1+\ldots+\mathbf{k}^{p}\\ -\mathbf{l}^1-\ldots-\mathbf{l}^{q}=\mathbf{k}}} \sum_{\substack{k\in\Zc^p,\, l\in\Zc^{q}\\ \Mc(k,l)=j} }P_{j,k,l} z^{\mathbf{k}^1}_{k_1}\dotsm z^{\mathbf{k}^p}_{k_p} \overline{z}^{\mathbf{l}^1}_{l_1}\dotsm \overline{z}^{\mathbf{l}^{q}}_{l_{q}}
\end{align*}
with the coefficients $P_{j,k,l}$ of (\ref{eq:nonlin}).
Modulated Fourier expansions can hence be seen as embedding the original system of equations in a larger system. The nonlinearity is the partial derivative with respect to $\overline{z}_{j}^{\mathbf{k}}$ of the modulation potential
$$
\mathcal{P}(\mathbf{z},\overline{\mathbf{z}})=
\sum_{p+q\ge 3} \sum_{\substack{\mathbf{k}^1+\ldots+\mathbf{k}^{p}\\ -\mathbf{l}^1-\ldots-\mathbf{l}^{q+1}=\mathbf{0}}} \sum_{\substack{k\in\Zc^p,\, l\in\Zc^{q+1}\\ \Mc(k,l)=0}} H_{kl} z^{\mathbf{k}^1}_{k_1}\dotsm z^{\mathbf{k}^p}_{k_p} \overline{z}^{\mathbf{l}^1}_{l_1}\dotsm \overline{z}^{\mathbf{l}^{q+1}}_{l_{q+1}}
$$
  with $\mathbf{z} = (z_j^{\mathbf{k}})_{j,\mathbf{k}}$ and   with the coefficients $H_{kl}$ of the non-quadratic part of the Hamiltonian $H(\xi,\overline{\xi})$. The modulation potential is invariant under transformations $z^{\mathbf{k}}_j\mapsto e^{i\mathbf{k}(\ell)\theta}z^{\mathbf{k}}_j$ for $\theta\in\R$ and fixed $\ell\in\Zc$.
The modulation system thus inherits the Hamiltonian structure from the original equations of motions for $\xi$, and its transformation invariance leads to  formal invariants
$$
\mathcal{I}_{\ell}(\mathbf{z},\overline{\mathbf{z}}) = \sum_{j,\mathbf{k}} \mathbf{k}(\ell) |z^{\mathbf{k}}_{j}|^2, \qquad \ell\in\Zc,
$$
of the modulation system, see \cite[Sect. 3.1]{Gau10}.   These formal invariants form the cornerstone for the study of long time intervals. 

On a short time interval of length $\varepsilon^{-1}$ it is possible to construct an approximate solution of the modulation system in an iterative way, such that---under certain assumptions to be verified below---the ansatz \eqref{eq:mfe} describes the solution $\xi$ up to a small error,
\[
\Norm{\xi(t)-\widetilde{\xi}(t)}{s} \le C\varepsilon^{N+3} \quad\text{for $0\le t\le c\varepsilon^{-1}$}
\]
with $\widetilde{\xi}$ given by \eqref{eq:mfe} with the approximate solution of the modulation system, see \cite[Sects. 3.2--3.4]{Gau10}. The constants depend on $N$ and $s$ but not on $\varepsilon$. Along this approximate solution $\mathbf{z}$ of the modulation system, the formal invariants $\mathcal{I}_{\ell}$ then become 
almost-invariants,
\[
\sum_{\ell\in\Zc} |\ell|^{2s} \Bigl|\frac{d}{d t} \mathcal{I}_\ell(\mathbf{z}(\varepsilon t),\overline{\mathbf{z}(\varepsilon t)})\Bigr| \le C\varepsilon^{N+2},
\]
which are close to the actions $I_{\ell}(\xi,\overline{\xi})=|\xi_\ell|^2$,
\[
\sum_{\ell\in\Zc} |\ell|^{2s} \bigl| \mathcal{I}_\ell(\mathbf{z}(\varepsilon t),\overline{\mathbf{z}(\varepsilon t)}) - I_\ell(\xi(t),\overline{\xi(t)}) \bigr| \le C\varepsilon^{\frac{5}{2}}.
\]
These almost-invariants allow us to repeat the construction of modulated Fourier expansions on short time intervals of length $\varepsilon^{-1}$ and patch $\varepsilon^{-N+1}$ of those short intervals together, see \cite[Sect. 3.5]{Gau10}. On a long time interval of length $\varepsilon^{-N}$ we then get near-conservation of actions as stated in Theorem~\ref{theorem:superactions} (with actions instead of super-actions).

Compared to the above description, the modulated Fourier expansion for our problem \eqref{eq:hammotion} has some subtleties that are caused by the partial resonances $\omega_j=\omega_{j'}$ for $|j|=|j'|$. Since all sums in the nonlinearity of \eqref{eq:hammotion} involve only products of the form $\xi_{k_1}\dotsm\xi_{k_p}\overline{\xi}_{l_{1}}\dotsm\overline{\xi}_{l_{q}}$ with $k_1+\dots+k_p-l_{1}-\dots-l_{q}=j$ (by the zero momentum condition in the Hamiltonian), only modulation functions $z^{\mathbf{k}}_j$ with
\[
j = j(\mathbf{k}) = \sum_{\ell\in\Zc} \mathbf{k}(\ell)\ell
\] 
can be different from zero. Moreover, since the frequencies $\omega_j$ in \eqref{eq:hammotion} are partially resonant, $\omega_{j}=\omega_{j'}$ for $|j|=|j'|$, we can distinguish exponentials $e^{-i(\mathbf{k}^1\cdot\bfomega)t}$ and $e^{-i(\mathbf{k}^2\cdot\bfomega)t}$ only if 
\[
\mathbf{k}^1 - \mathbf{k}^2 \not\in \Bigl\{\, \mathbf{k} : \sum_{|\ell|^2=n} \mathbf{k}(\ell) = 0 \text{ for all $n\in\N$} \,\Bigr\}.
\]
For this reason, the sum in \eqref{eq:hammotion} is in our situation only over a set of representatives of sequences $\mathbf{k}$ where $j(\mathbf{k})$ or $\mathbf{k}\cdot\bfomega$ are distinguishable (in the above sense). The main consequence is that the quantities $\mathcal{I}_\ell$ from above are no longer invariants of the modulation system, but only certain sums of them:
\[
\mathcal{J}_n(\mathbf{z},\overline{\mathbf{z}})= \sum_{\ell\in\Zc : |\ell|^2=n} \mathcal{I}_\ell(\mathbf{z},\overline{\mathbf{z}}), \qquad n\in\N.
\]
Along the approximate solution of the modulation system, they are close to the corresponding sums of the actions $I_\ell$ , the super-actions $J_n$. In this way we get long-time near-conservation of super-actions as in Theorem~\ref{theorem:superactions}. 

We finally state and verify the assumptions needed for the iterative construction of modulation functions. The first lemma below summarises the assumptions on the nonlinearity in \eqref{eq:hammotion}, whereas the second lemma below deals with the non-resonance condition on the frequencies describing the linear part of \eqref{eq:hammotion}. The properties stated in these lemmas are precisely the assumptions under which Theorem~\ref{theorem:superactions} has been shown in \cite[Theorem~2.7]{Gau10}.

\begin{lemma}\label{lemma:mfe1}
The expansion \eqref{eq:nonlin} of the nonlinearity in \eqref{eq:hammotion} has the following properties.
\begin{enumerate}
\item[(i)] It fulfills the zero momentum condition
$$
P_{j,k,l}=0 \quad\text{if \ $j \ne \Mc(k,l)$}
$$
for $j\in\Zc$, $k\in\Zc^p$ and $l\in\Zc^{q}$.
\item[(ii)] There exist constants $C_{p,q,s}$ depending only on $p$, $q$, $s$ and $\rho$ such that for
$$
|P|_{j}^{p,q}(\xi^1,\dots,\xi^p,\overline{\xi}^1,\dots,\overline{\xi}^{q}) = \sum_{k\in\Zc^p,\, l\in\Zc^{q}} |P_{j,k,l}| \,\xi^1_{k_1}\dots \xi^p_{k_p} \overline{\xi}^1_{l_1}\dotsm \overline{\xi}^{q}_{l_{q}}
$$
the estimate
\begin{equation}\label{eq:algebra}%\begin{split}
\Norm{|P|^{p,q}(\xi^1,\dots,\xi^p,\overline{\xi}^1,\dots,\overline{\xi}^{q})}{s}
 \le C_{p,q,s} \Norm{\xi^1}{s}\dotsm\Norm{\xi^p}{s} \Norm{\overline{\xi}^1}{s}\dotsm \Norm{\overline{\xi}^{q}}{s}
%\end{split}
\end{equation}
holds for $\xi^1,\dots,\xi^p,\overline{\xi}^1,\dots,\overline{\xi}^{q}\in \ell^2_s$ if $s>\frac{d}{2}$.
\item[(iii)] There exist  $r_0>0$ depending only on $\rho$, and $C_s$ depending in addition on $s>\frac{d}{2}$ such that
$$
\sum_{p+q\ge2} C_{p,q,s}|z|^{p+q-2}\le C_s \quad\text{for all $z\in\C$ with $|z|\le r_0$.}
$$
\end{enumerate}
\end{lemma}
\begin{proof}
Property (i) is obvious, see \eqref{eq:nonlin}. 

For property (ii) we recall that \eqref{eq:algebra} was verified in \cite[Subsect.~2.6]{Gau10} in the situation $P_{j,k,l}=0$ for $j\ne \Mc(k,l)$ and $P_{j,k,l}=1$ else. The proof is just a repeated application of the Cauchy-Schwarz inequality, and the corresponding constants $C_{p,q,s}$ are given by $C^{p+q}$ with $C$ depending only on $s$. In our situation here, the coefficients $P_{j,k,l}$ vanish for $j\ne \Mc(k,l)$ and  can be bounded with Lemma~\ref{lem:H}. This implies that the second property (ii) is satisfied with constants $C_{p,q,s} = (q+1) M L^{p+q+1} C^{p+q}$.

These constants satisfy (iii) with $r_0$ and $C_s$ depending only on $\rho$ and $s$. 
\end{proof}

\begin{lemma}\label{lemma:mfe2}
The frequencies $\Omega_n$ grow like $n$, $c_1n \le \Omega_n \le C_1n$ with positive constants $c_1$ and $C_1$ depending only on $\rho$.

Moreover, for all $\rho_0>0$ such that $1 + 2 \lambda \rho_0^2 > 0$ and for all positive integer $N$, there exist $s_0$ and a set of full measure $\mathcal{P}$ in the interval $(0,\rho_0]$ such that for all $s \geq s_0$ and all $\rho \in \Pc$ the following non-resonance condition holds: There exist $\varepsilon_0>0$ and $C_0$ such that for all $r\le 2N+2d+6$ and all $0<\varepsilon\le\varepsilon_0$,
$$
\Bigl( \frac{n}{n_1 \dotsm n_r} \Bigr)^{s-\frac{d+1}{2}} \, \varepsilon^{r} \le C_0 \,\varepsilon^{2N+2d+8}
$$
whenever a near-resonance
$$
|\Omega_{n} \pm \Omega_{n_1} \pm \dots \pm \Omega_{n_r}| < \varepsilon^{\frac{1}{2}}
$$
occurs with frequencies that do not cancel pairwise.
\end{lemma}
\begin{proof}
The asymptotic growth behaviour of the frequencies is obvious, and the non-resonance condition is implied by the non-resonance condition of Lemma~\ref{lem:non-res} as shown in \cite[Lemma 1]{CHL08}.
\end{proof}

\subsection{Proof of the orbital stability (\ref{eq:orb})}
We have
\[
\inf_{\varphi\in\mathbb{R}} \| e^{-i m \cdot \bullet}u(\bullet,t) -  e^{i\varphi} u_{m}(0) \|_{H^s}^2 = \bigl| |u_m(t)| - |u_m(0)| \bigr|^2 + \| e^{-i m \cdot \bullet}u(\bullet,t) -  u_{m}(t) \|_{H^s}^2,
\]
and by the conservation of the $L^2$ norm
\begin{align*}
\bigl| |u_m(t)| - |u_m(0)| \bigr|^2 &\le \bigl| |u_m(t)|^2 - |u_m(0)|^2 \bigr| = \Bigl| \sum_{m\ne j\in\mathbb{Z}^d} |u_j(0)|^2 - \sum_{m\ne j\in\mathbb{Z}^d} |u_j(t)|^2 \Bigr|\\
 &\le \max \bigl( \| e^{-i m \cdot \bullet}u(\bullet,0) -  u_{m}(0) \|_{L^2}^2 , \| e^{-i m \cdot \bullet}u(\bullet,t) -  u_{m}(t) \|_{L^2}^2 \bigr).
\end{align*}
The estimate (\ref{eq:orb}) thus follows from Theorem~\ref{theorem:main}.

\end{document}